\documentclass{amsart}

\usepackage[english]{babel}
\usepackage[utf8x]{inputenc}
\usepackage[T1]{fontenc}
\usepackage{lmodern}
\usepackage{multicol}
\usepackage{comment}

\usepackage{amsmath}
\usepackage{amsthm}

\newtheorem{definition}{Definition}[section]
\newtheorem{theorem}{Theorem}[section]
\newtheorem{corollary}{Corollary}[theorem]
\newtheorem{proposition}{Proposition}[section]
\newtheorem{lemma}[theorem]{Lemma}
\newtheorem*{remark}{Remark}
\newtheorem{example*}{Example}

\newcommand{\R}{\mathbb{R}} 
\newcommand{\Z}{\mathbb{Z}} \newcommand{\N}{\mathbb{N}} 
\newcommand{\C}{\mathbb{C}}

\newcommand{\HH}{\mathcal{H}}
\newcommand{\KK}{\mathcal{K}}

\usepackage{textcase}
\title{Pseudo-Riesz Bases}
\author{Deborpita Biswas}
\author{Mishko Mitkovski}
\address[Mishko Mitkovski]{School of Mathematical and Statistical Sciences,
Clemson University, 220 Parkway Dr., Clemson, SC 29630, USA}
\email{mmitkov@clemson.edu}
\thanks{Mishko Mitkovski is supported in part by the National Science
Foundation under grant \mbox{DMS-2453810}.}
\date{August 2024}

\begin{document}

\begin{abstract}
We develop a Fredholm framework for converting Bessel systems into Riesz
bases by finite modifications. We call such systems pseudo-Riesz bases
and characterize them by Fredholmness of their synthesis operators.
The kernel and cokernel dimensions determine the optimal numbers of
deletions and adjunctions, while the corresponding semi-Fredholm
conditions describe the one-sided modification problems. For pseudo-Riesz
bases, we establish automatic Besselness and unconditional convergence for alternative duals
and synthesis pseudo-duals that reconstruct every vector modulo a bounded
finite-rank operator. Every such dual is itself pseudo-Riesz and has the
same synthesis index. Relative and square-summable perturbation theorems
provide stability results for these classes. For pseudo-Riesz systems of
distinct real exponentials, the
Hru\v{s}\v{c}ev--Nikol'skii--Pavlov model and Toeplitz index theory yield
a precise correction rule: positive index $p$ allows deletion of any
$p$ original terms, negative index $-q$ allows adjunction of any $q$
distinct new real-frequency exponentials, and index zero characterizes
Riesz bases. Thus optimal finite corrections can always be performed
within the exponential family. As a consequence, we obtain an asymptotic
form of Kadec's $1/4$-theorem that permits finitely many arbitrary
frequency deviations, provided distinctness is preserved.
\end{abstract}

\maketitle

\section{Introduction}

Let $\HH$ be a separable Hilbert space, and let
$F=(f_n)_{n\in\N}$ be a Bessel sequence in $\HH$. Its synthesis operator is
\[
    S_F:\ell^2(\N)\longrightarrow\HH,
    \qquad
    S_F((c_n)_{n\in\N})=\sum_{n\in\N}c_nf_n.
\]
Thus, $F$ is a Riesz basis for $\HH$ precisely when $S_F$ is an
isomorphism. This observation leads to a natural finite-modification
question: when can a Bessel sequence be transformed into a Riesz basis by
deleting finitely many terms and adjoining finitely many vectors?

The frame-theoretic instance of this question is classical. A
\emph{near-Riesz basis}, in the sense of Holub, is a frame from which
finitely many terms can be deleted to obtain a Riesz basis. Holub proved
that a frame is a near-Riesz basis if and only if its synthesis operator is
surjective and has finite-dimensional kernel; moreover, the dimension of
that kernel gives the number of terms that must be deleted
\cite{holub1994bases}. Casazza and Christensen established perturbation
results for near-Riesz bases \cite{casazza1998frames}. More recently, Heil
and Yu showed that a frame is a near-Riesz basis precisely when every
alternative dual, equivalently every synthesis pseudo-dual, is a frame.
When the frame has at most finitely many zero elements, they also
characterized this property by unconditional convergence of the associated
reconstruction series for every such dual \cite[Theorem~3.10 and
Corollary~3.8]{heil2023}. These duality characterizations provide an
important benchmark for the broader class considered below.

More generally, the separate questions of finite redundancy and finite
deficiency have an established literature. Balan, Casazza, Heil, and
Landau developed the notions of excess and deficit and related them to the
analysis and synthesis operators of Bessel sequences and frames
\cite{balan2003deficits}. Baki\'c and Beri\'c gave an operator-theoretic
characterization of Bessel sequences that admit a finite extension to a
frame \cite{bakic2015finite}. The present paper does not claim to initiate
the theory of finite extensions or of deficit and excess. Its focus is the
simultaneous finite-redundancy and finite-deficiency problem, with a Riesz
basis as the target.

Accordingly, we call $F$ a \emph{pseudo-Riesz basis} if there are a finite
set $\sigma\subset\N$ and finitely many vectors
$g_1,\ldots,g_m\in\HH$ such that
\[
    \{f_n:n\notin\sigma\}\cup\{g_1,\ldots,g_m\}
\]
is a Riesz basis for $\HH$. Finite replacement is included in this
definition, since a replacement can be performed by deleting the old term
and adjoining the new one. This class contains both near-Riesz bases and
Riesz sequences whose closed span has finite codimension. It also contains
Bessel sequences that are neither frames nor Riesz sequences. For example,
if $(e_n)_{n\in\N}$ is an orthonormal basis for $\HH$, then
\[
    (e_1,e_1,e_3,e_4,\ldots)
\]
becomes an orthonormal basis after one deletion and one adjunction, although
it is neither a frame for $\HH$ nor a Riesz sequence.

The central structural result identifies this finite-modification property
with Fredholmness of the synthesis operator:
\[
    F\text{ is a pseudo-Riesz basis}
    \quad\Longleftrightarrow\quad
    S_F:\ell^2(\N)\to\HH\text{ is Fredholm}.
\]
This formulation keeps the two obstructions visible. If $S_F$ is
Fredholm, then $\dim\ker S_F$ is the least number of terms that must be
deleted to obtain a Riesz sequence, while
$\operatorname{codim}\operatorname{ran}S_F$ is the least number of vectors
that must be adjoined to complete the resulting sequence to a Riesz basis.
In particular,
\[
    \operatorname{ind}S_F
    =\dim\ker S_F-\operatorname{codim}\operatorname{ran}S_F
\]
records the imbalance between these two finite operations. Holub's
near-Riesz bases are exactly the surjective case of this picture.

The one-sided versions of this result identify pseudo-Riesz sequences and
Bessel sequences admitting finite frame completions with the left and right
semi-Fredholm synthesis operators, respectively. We also obtain
coefficient reconstruction modulo finite rank and perturbation
results for these three classes.

The connection with Heil--Yu's duality theory also has a positive
finite-defect form. Suppose that $F$ is a pseudo-Riesz basis and that an
arbitrary sequence $G=(g_n)$ satisfies either
\[
    (I-K)x=\sum_n\langle x,g_n\rangle f_n
    \quad\text{or}\quad
    (I-K)x=\sum_n\langle x,f_n\rangle g_n,
    \qquad x\in\HH,
\]
in norm, for a bounded finite-rank operator $K$. Then $G$ is automatically
Bessel, is itself a pseudo-Riesz basis, and has the same synthesis index
as $F$; the series consequently converge unconditionally. Theorem~
\ref{thm:automatic-besselness} proves this without any Bessel assumption
on $G$. Reconstruction on all of $\HH$, rather than only on the closed
span of $F$, is essential.

The analogous questions for exponential systems are particularly subtle.
Seip showed that the unrestricted reduction and extension questions are
negative: there are real exponential frames with no Riesz-basis subsystem
and real exponential Riesz sequences with no exponential Riesz-basis
extension \cite[Theorems~2.7--2.8]{seip1995exponential}.  His examples have
infinite excess or deficiency in Seip's sense and hence lie outside the
finite-defect Fredholm regime considered here.  They explain why the
finiteness hypothesis is essential rather than merely technical.

For distinct real frequencies, our finite-defect conclusion requires no
separate separation hypothesis. Indeed, a pseudo-Riesz exponential system
becomes a Riesz sequence after finitely many deletions, which forces its
entire frequency set to be uniformly separated. Lemma~
\ref{lem:horizontal-carleson} then supplies the shifted Carleson condition.
Consequently, Theorem~\ref{thm:real-exponential-classification} gives the
following classification for every pseudo-Riesz basis $E_\Lambda$ of
distinct real exponentials in $L^2(0,1)$, with
$k=\operatorname{ind}S_{E_\Lambda}$:
\begin{enumerate}
    \item If $k=p>0$, deleting any $p$ terms gives a Riesz basis.
    \item If $k=-q<0$, adjoining any $q$ distinct new real-frequency
    exponentials gives a Riesz basis.
    \item If $k=0$, the original system is a Riesz basis.
\end{enumerate}
Thus, whenever finite deletions and arbitrary Hilbert-space adjunctions
can produce a Riesz basis, an optimal correction can be made entirely
within the exponential family, using only deletions or only adjunctions.
Unlike general pseudo-Riesz bases, these exponential systems cannot have
both nonzero excess and nonzero deficit.

Sections~2--5 develop this operator-theoretic framework by separating the
finite-completion and finite-reduction problems and then combining them in
the Fredholm characterization above. Section~6 studies stability under
precisely stated relative and Hilbert--Schmidt perturbations.
Sections~7--8 use the Hru\v{s}\v{c}ev--Nikol'skii--Pavlov model
\cite{hruscev1981} to intertwine exponential synthesis with a model-space
projection and the Toeplitz operator $T_{\overline{B}\Theta}$. We retain
the explicit shifted Carleson hypothesis for the general one-sided
operator correspondences. In the pseudo-Riesz case it is automatic, and
Coburn's alternative and finite Blaschke factors yield the classification
above. The finite-defect hypothesis remains essential; no such conclusion
is asserted for arbitrary exponential frames or Riesz sequences.

\section{Preliminaries}

Throughout, all Hilbert spaces are complex and separable. We write
$\ell^2=\ell^2(\N)$, and $(e_n)_{n\in\N}$ denotes the standard orthonormal
basis of $\ell^2$.

\subsection{Bessel sequences, frames, and Riesz sequences}

Let $F=(f_n)_{n\in\N}$ be a sequence in a Hilbert space $\HH$. We call $F$
a \emph{Bessel sequence} if there is a constant $B>0$ such that
\[
    \sum_{n\in\N} |\langle f,f_n\rangle|^2
    \leq B\|f\|^2,
    \qquad f\in\HH.
\]
Equivalently, the map initially defined on $c_{00}$ by
\[
    (c_n)_{n\in\N}\longmapsto \sum_{n\in\N}c_nf_n
\]
extends to a bounded operator
\[
    S_F:\ell^2\longrightarrow\HH,
    \qquad
    S_F(c_n)=\sum_{n\in\N}c_nf_n.
\]
We call $S_F$ the \emph{synthesis operator} of $F$. Its adjoint
\[
    C_F:=S_F^*:\HH\longrightarrow\ell^2,
    \qquad
    C_Ff=(\langle f,f_n\rangle)_{n\in\N},
\]
is the \emph{analysis operator}. Notice that
\[
    \overline{\operatorname{ran}S_F}
    =
    \overline{\operatorname{span}}\{f_n:n\in\N\}.
\]

Set
\[
    M_F:=\overline{\operatorname{span}}\{f_n:n\in\N\}.
\]
The sequence $F$ is a \emph{frame sequence} if it is a frame for $M_F$;
that is, if there is a constant $A>0$ such that
\[
    A\|f\|^2
    \leq
    \sum_{n\in\N}|\langle f,f_n\rangle|^2,
    \qquad f\in M_F.
\]
If $M_F=\HH$, then $F$ is a \emph{frame} for $\HH$. The following standard
characterizations will be used repeatedly:
\[
\begin{aligned}
    F\text{ is a frame for }\HH
    &\quad\Longleftrightarrow\quad
    S_F\text{ is surjective},\\
    F\text{ is a frame sequence}
    &\quad\Longleftrightarrow\quad
    \operatorname{ran}S_F\text{ is closed}.
\end{aligned}
\]
When $F$ is a frame for $\HH$, its frame operator is
\[
    \mathsf{S}_F=S_FS_F^*:\HH\longrightarrow\HH.
\]
It is positive and boundedly invertible, and
$(\mathsf{S}_F^{-1}f_n)_{n\in\N}$ is the canonical dual frame.

A Bessel sequence $F$ is a \emph{Riesz sequence} if there is a constant
$A>0$ such that
\[
    A\sum_{n\in\N}|c_n|^2
    \leq
    \left\|\sum_{n\in\N}c_nf_n\right\|^2,
    \qquad
    (c_n)\in\ell^2.
\]
Equivalently,
\[
\begin{aligned}
    F\text{ is a Riesz sequence}
    &\quad\Longleftrightarrow\quad
    S_F\text{ is bounded below}\\
    &\quad\Longleftrightarrow\quad
    S_F\text{ is injective and has closed range}.
\end{aligned}
\]
Thus, $F$ is a Riesz basis for $\HH$ if and only if
$S_F:\ell^2\to\HH$ is a bounded isomorphism.

\subsection{Closed-range and Fredholm operators}

Let $T\in\mathcal{B}(\HH,\KK)$. The closed-range theorem gives the
following equivalent conditions:
\[
\begin{aligned}
    \operatorname{ran}T\text{ is closed}
    &\quad\Longleftrightarrow\quad
    \text{there is a constant }\gamma>0\text{ such that}\\
    &\hspace{2.3cm}
    \|Tx\|\geq\gamma\,
    \operatorname{dist}(x,\ker T)
    \quad\text{for every }x\in\HH.
\end{aligned}
\]
When these conditions hold,
\[
    T\big|_{(\ker T)^\perp}:
    (\ker T)^\perp\longrightarrow\operatorname{ran}T
\]
is an isomorphism of Hilbert spaces. In particular, a Bessel sequence $F$
is a frame sequence if and only if its synthesis operator satisfies the
displayed distance-to-kernel estimate.

The operator $T$ is called \emph{left semi-Fredholm} if
$\operatorname{ran}T$ is closed and $\dim\ker T<\infty$. It is called
\emph{right semi-Fredholm} if $\operatorname{ran}T$ is closed and
$\operatorname{codim}\operatorname{ran}T<\infty$. It is
\emph{semi-Fredholm} if it is either left or right semi-Fredholm. It is
\emph{Fredholm} if it is both left and right semi-Fredholm. For a Fredholm
operator,
\[
    \operatorname{ind}T
    :=
    \dim\ker T-\operatorname{codim}\operatorname{ran}T
    =
    \dim\ker T-\dim\ker T^*.
\]

We shall also use the standard parametrix characterization: $T$ is
Fredholm if and only if there is an operator
$R\in\mathcal{B}(\KK,\HH)$ such that
\[
    RT-I_{\HH}
    \quad\text{and}\quad
    TR-I_{\KK}
\]
are compact; equivalently, they may be chosen to be finite-rank. Moreover,
compact perturbations of a Fredholm operator are Fredholm with the same
index, and
\[
    \operatorname{ind}(T^*)=-\operatorname{ind}(T).
\]

\subsection{Finite-coordinate complements}

We shall use a simple coordinate lemma. It converts finite-dimensional
kernels into finite deletions of sequence elements.

\begin{lemma}\label{lem:coordinate-complement}
Let $E\subset\ell^2$ be a subspace with $\dim E=d<\infty$. Then there is a
finite set $\sigma\subset\N$ with $|\sigma|=d$ such that the coordinate
projection
\[
    P_\sigma|_E:E\longrightarrow\mathbb{C}^{\sigma}
\]
is an isomorphism.
\end{lemma}

\begin{proof}
The coordinate functionals separate points of $E$. Hence one can choose
$d$ coordinate functionals whose restrictions to $E$ form a basis of
$E^*$. Their indices form a set $\sigma$ for which $P_\sigma|_E$ is
injective. Since its domain and codomain have the same finite dimension,
it is an isomorphism.
\end{proof}

For such a set $\sigma$, let
\[
    M_\sigma
    :=
    \{c\in\ell^2:c_j=0\text{ for every }j\in\sigma\}.
\]
Then $E\cap M_\sigma=\{0\}$ and
$\ell^2=E\oplus M_\sigma$. Consequently, if
$T:\ell^2\to\HH$ has closed range and $\dim\ker T=d$, then $\sigma$ may be
chosen so that
\[
    T|_{M_\sigma}:M_\sigma\longrightarrow\operatorname{ran}T
\]
is an isomorphism.

\subsection{Near-Riesz bases}

A frame $F=(f_n)_{n\in\N}$ for $\HH$ is a \emph{near-Riesz basis} if there
is a finite set $\sigma\subset\N$ such that
$(f_n)_{n\notin\sigma}$ is a Riesz basis for $\HH$. Holub proved that $F$
is a near-Riesz basis if and only if $\dim\ker S_F<\infty$; in that case,
whenever $(f_n)_{n\notin\sigma}$ is a Riesz basis, one has
\[
    |\sigma|=\dim\ker S_F.
\]
See~\cite{holub1994bases}.

Holub also related near-Riesz bases to unconditional convergence of
scalar series. A frame is called \emph{unconditional} if every
norm-convergent series $\sum_n c_nf_n$ converges unconditionally. For
frames bounded below in norm, this property is equivalent to being a
near-Riesz basis \cite{holub1994bases}. The result of Heil and Yu below
characterizes near-Riesz bases through unconditional convergence of
reconstruction series for every alternative dual or synthesis pseudo-dual
\cite{heil2023}.

For a frame $F=(f_n)_{n\in\N}$, a sequence $G=(g_n)_{n\in\N}$ is an
\emph{alternative dual} of $F$ if
\[
    f=\sum_{n\in\N}\langle f,g_n\rangle f_n,
    \qquad f\in\HH,
\]
and $G$ is a \emph{synthesis pseudo-dual} of $F$ if
\[
    f=\sum_{n\in\N}\langle f,f_n\rangle g_n,
    \qquad f\in\HH.
\]
Here the defining series are required to converge in norm; the sequences
$G$ need not a priori be Bessel sequences. The following theorem is the
duality characterization of near-Riesz bases that will be relevant for
comparison with pseudo-Riesz bases.

\begin{theorem}[Heil--Yu]
\label{thm:heil-yu}
Let $F$ be a frame for $\HH$. The following are equivalent:
\begin{enumerate}
    \item $F$ is a near-Riesz basis;
    \item every alternative dual of $F$ is a frame;
    \item every synthesis pseudo-dual of $F$ is a frame.
\end{enumerate}
If $F$ has at most finitely many zero terms, these conditions are also
equivalent to each of the following unconditional-convergence properties:
\begin{enumerate}
    \item[(a)] for every alternative dual $G=(g_n)$ of $F$ and every
    $f\in\HH$, the series $\sum_n\langle f,g_n\rangle f_n$ converges
    unconditionally;
    \item[(b)] for every synthesis pseudo-dual $G=(g_n)$ of $F$ and every
    $f\in\HH$, the series $\sum_n\langle f,f_n\rangle g_n$ converges
    unconditionally.
\end{enumerate}
\end{theorem}

Exact reconstruction on $\HH$ is unavailable for an incomplete
pseudo-Riesz basis. The appropriate replacement is reconstruction on all
of $\HH$ modulo a bounded finite-rank error. Theorem~\ref{thm:automatic-besselness} extends the automatic Besselness conclusion
to this setting, while Corollary~\ref{cor:finite-defect-heil-yu} gives the
corresponding characterization among sequences admitting finite frame
completions.

\section{Finite frame completions}
The problem of extending a Bessel sequence by finitely many vectors
to a frame is a standard finite-deficiency question; see, for example,
\cite{bakic2015finite}. We record the synthesis-operator formulation that
will be used below.

\begin{definition}
A Bessel sequence $F=(f_n)_{n\in\N}$ in $\HH$ is said to \emph{admit a
finite frame completion} if there are vectors $h_1,\ldots,h_m\in\HH$ such
that
\[
    (h_1,\ldots,h_m,f_1,f_2,\ldots)
\]
is a frame for $\HH$.
\end{definition}

\begin{proposition}\label{prop:finite-frame-completion}
Let $F=(f_n)_{n\in\N}$ be a Bessel sequence in $\HH$, with synthesis
operator $S_F$. The following are equivalent:
\begin{enumerate}
    \item $F$ admits a finite frame completion;
    \item $\operatorname{ran}S_F$ is closed and
    $\operatorname{codim}\operatorname{ran}S_F<\infty$;
    \item $S_F$ is right semi-Fredholm.
\end{enumerate}
Moreover, when these conditions hold, the least number of vectors in a
finite frame completion is
\[
    \operatorname{codim}\operatorname{ran}S_F
    =
    \dim\ker S_F^*.
\]
\end{proposition}

\begin{proof}
Assume first that $h_1,\ldots,h_m$ give a finite frame completion, and
define
\[
    J:\C^m\longrightarrow\HH,
    \qquad
    J(a_1,\ldots,a_m)=\sum_{j=1}^m a_jh_j.
\]
The extended synthesis operator
\[
    \widetilde S:\C^m\oplus\ell^2\longrightarrow\HH,
    \qquad
    \widetilde S(a,c)=Ja+S_Fc,
\]
is surjective. Hence it has a bounded right inverse
$L=(L_0,L_1):\HH\to\C^m\oplus\ell^2$. Thus
\[
    S_FL_1=I_{\HH}-JL_0.
\]
The subspace $E:=\ker(JL_0)$ is closed and has finite codimension, and the
displayed identity shows that $E\subset\operatorname{ran}S_F$. Every
subspace containing a closed finite-codimensional subspace is itself closed
and finite-codimensional. Therefore $\operatorname{ran}S_F$ has the
properties in (ii).

Conversely, suppose that $\operatorname{ran}S_F$ is closed and has
codimension $d$. Choose an orthonormal basis
$h_1,\ldots,h_d$ for $(\operatorname{ran}S_F)^\perp$. The synthesis
operator of
\[
    (h_1,\ldots,h_d,f_1,f_2,\ldots)
\]
is surjective, and hence this Bessel sequence is a frame. The equivalence
of (ii) and (iii) is the definition of right semi-Fredholmness.

Finally, any completion by $m$ vectors induces a surjective map from a
space of dimension at most $m$ onto
$\HH/\operatorname{ran}S_F$. Hence
$m\geq\operatorname{codim}\operatorname{ran}S_F$, while the construction
above attains this number.
\end{proof}

\begin{remark}
The closed-range hypothesis in Proposition~\ref{prop:finite-frame-completion}
is essential: finite deficiency of the closed span alone does not guarantee
a finite frame completion; see \cite[Example~2.1]{bakic2015finite}.
\end{remark}

\subsection{Finite-rank essential duality}

Baki\'c and Beri\'c call two Bessel sequences $F$ and $G$ essentially dual
when $I_{\HH}-S_GC_F$ is compact
\cite[Definition~2.6]{bakic2015finite}. Taking adjoints shows that this is
a symmetric relation. We use only its finite-rank specialization.

\begin{definition}
Let $F=(f_n)_{n\in\N}$ and $G=(g_n)_{n\in\N}$ be Bessel sequences in
$\HH$. We say that $F$ and $G$ are \emph{finite-rank essentially dual} if
\[
    I_{\HH}-S_GC_F
\]
has finite rank. Equivalently, $S_FC_G-I_{\HH}$ has finite rank.
\end{definition}

Since $F$ and $G$ are Bessel, the series
\[
    \sum_{n\in\N}\langle f,g_n\rangle f_n
\]
converges in norm for every $f\in\HH$. The next elementary observation
identifies finite-rank essential duality with reconstruction modulo a
finite-dimensional defect.

\begin{lemma}\label{lem:finite-defect-reconstruction}
Let $F$ and $G$ be Bessel sequences in $\HH$. The following are equivalent:
\begin{enumerate}
    \item There is a closed finite-codimensional subspace $M\subset\HH$
    such that
    \[
        f=\sum_{n\in\N}\langle f,g_n\rangle f_n,
        \qquad f\in M.
    \]
    \item $S_FC_G-I_{\HH}$ has finite rank.
\end{enumerate}
\end{lemma}

\begin{proof}
Set $K:=S_FC_G-I_{\HH}$. If the reconstruction formula holds on $M$, then
$K$ vanishes on $M$. Since $M$ has a finite-dimensional complement, it
follows that $\operatorname{ran}K$ is finite-dimensional. Conversely, if
$K$ has finite rank, then $\ker K$ is a closed finite-codimensional
subspace on which the reconstruction formula holds.
\end{proof}

\begin{proposition}\label{prop:finite-rank-essential-dual}
A Bessel sequence $F$ admits a finite frame completion if and only if there
is a Bessel sequence $G$ that is finite-rank essentially dual to $F$. In
that case, every such $G$ also admits a finite frame completion.
\end{proposition}

\begin{proof}
Suppose first that $F$ and $G$ are finite-rank essentially dual, and put
\[
    K:=S_FC_G-I_{\HH}.
\]
Then $\ker K$ is closed and finite-codimensional, and
\[
    \ker K\subset\operatorname{ran}S_F.
\]
Proposition~\ref{prop:finite-frame-completion} therefore shows that $F$
admits a finite frame completion. Taking adjoints gives
\[
    S_GC_F-I_{\HH}=K^*.
\]
The same argument shows that $G$ admits a finite frame completion.

Conversely, suppose that $F$ admits a finite frame completion and set
\[
    A:=S_FC_F.
\]
By Proposition~\ref{prop:finite-frame-completion}, $A$ has closed range
$\operatorname{ran}S_F$. Its Moore--Penrose inverse $A^\dagger$ is bounded
and self-adjoint. Define
\[
    g_n:=A^\dagger f_n,
    \qquad n\in\N.
\]
Then $G=(g_n)_{n\in\N}$ is Bessel and
\[
    S_FC_G
    =
    S_FC_FA^\dagger
    =
    AA^\dagger
    =
    P_{\operatorname{ran}S_F}
    =
    I_{\HH}-P_{(\operatorname{ran}S_F)^\perp}.
\]
Thus $F$ and $G$ are finite-rank essentially dual.
\end{proof}

\begin{remark}
The sequence constructed in the proof is canonical: when $F$ is a frame,
it is precisely the canonical dual frame of $F$.
\end{remark}

\section{Pseudo-Riesz sequences}

Finite deletion addresses redundancy without requiring completeness in
$\HH$.

\begin{definition}
A Bessel sequence $F=(f_n)_{n\in\N}$ in $\HH$ is called a
\emph{pseudo-Riesz sequence} if there is a finite, possibly empty, set
$\sigma\subset\N$ such that
\[
    (f_n)_{n\notin\sigma}
\]
is a Riesz sequence.
\end{definition}

Thus, a pseudo-Riesz sequence need not be complete in $\HH$; after a finite
deletion, it is a Riesz basis for its own closed span.

\begin{proposition}\label{prop:pseudo-riesz-sequence}
Let $F=(f_n)_{n\in\N}$ be a Bessel sequence in $\HH$, with synthesis
operator $S_F$. The following are equivalent:
\begin{enumerate}
    \item $F$ is a pseudo-Riesz sequence;
    \item $\operatorname{ran}S_F$ is closed and
    $\dim\ker S_F<\infty$;
    \item $S_F$ is left semi-Fredholm.
\end{enumerate}
Moreover, when these conditions hold, the least number of terms that must
be deleted in order to obtain a Riesz sequence is
\[
    \dim\ker S_F.
\]
\end{proposition}

\begin{proof}
Assume first that $(f_n)_{n\notin\sigma}$ is a Riesz sequence for some
finite set $\sigma\subset\N$. Under the natural identification of
$\ell^2(\N\setminus\sigma)$ with $M_\sigma$, its synthesis operator is
$S_F|_{M_\sigma}$. Hence $S_F|_{M_\sigma}$ is bounded below and has closed
range. Since
\[
    \operatorname{ran}S_F
    =
    \operatorname{ran}(S_F|_{M_\sigma})
    +\operatorname{span}\{f_j:j\in\sigma\},
\]
$\operatorname{ran}S_F$ is closed. Moreover, the coordinate projection
$P_\sigma$ is injective on $\ker S_F$: indeed, if
$c\in\ker S_F$ and $P_\sigma c=0$, then $c\in M_\sigma$, and the
injectivity of $S_F|_{M_\sigma}$ gives $c=0$. Consequently,
\[
    \dim\ker S_F\leq |\sigma|.
\]

Conversely, suppose that $\operatorname{ran}S_F$ is closed and put
$d:=\dim\ker S_F<\infty$. By
Lemma~\ref{lem:coordinate-complement}, there is a finite set
$\sigma\subset\N$, with $|\sigma|=d$, for which
\[
    S_F|_{M_\sigma}:M_\sigma\longrightarrow\operatorname{ran}S_F
\]
is an isomorphism. Thus $(f_n)_{n\notin\sigma}$ is a Riesz sequence.
The preceding inequality proves the asserted minimality. The equivalence
of (ii) and (iii) is the definition of left semi-Fredholmness.
\end{proof}

\begin{remark}
For a pseudo-Riesz sequence, $\operatorname{ran}S_F$ is closed, so $F$ is
a frame for
\[
    M_F=\overline{\operatorname{span}}\{f_n:n\in\N\}
    =\operatorname{ran}S_F.
\]
Thus $F$ is a near-Riesz basis when viewed in $M_F$. Its minimal deletion
number is the usual frame excess relative to $M_F$; we retain the more
explicit phrase ``minimal deletion number'' here. See
\cite{holub1994bases,balan2003deficits}.
\end{remark}

\subsection{Finite-rank coefficient reconstruction}

The following is the coefficient-space analogue of the finite-rank
essential-duality relation from the preceding section. Its compact analogue
is discussed in \cite[Lemma~2.3 and Section~3]{bakic2015finite}.

\begin{lemma}\label{lem:finite-coordinate-reconstruction}
Let $F$ and $G$ be Bessel sequences in $\HH$. The following are equivalent:
\begin{enumerate}
    \item There is a closed finite-codimensional subspace
    $M\subset\ell^2$ such that
    \[
        C_GS_Fc=c,
        \qquad c\in M;
    \]
    \item $C_GS_F-I_{\ell^2}$ has finite rank.
\end{enumerate}
\end{lemma}

\begin{proof}
Set $K=C_GS_F-I_{\ell^2}$. If the first condition holds, then $K$ vanishes
on $M$ and therefore factors through the finite-dimensional quotient
$\ell^2/M$. Hence $K$ has finite rank. Conversely, if $K$ has finite rank,
then $\ker K$ is a closed finite-codimensional subspace on which
$C_GS_F=I_{\ell^2}$.
\end{proof}

\begin{proposition}\label{prop:finite-rank-coordinate-left-inverse}
Let $F$ be a Bessel sequence in $\HH$. The following are equivalent:
\begin{enumerate}
    \item $F$ is a pseudo-Riesz sequence;
    \item There is a Bessel sequence $G$ in $\HH$ such that
    \[
        C_GS_F-I_{\ell^2}
    \]
    has finite rank.
\end{enumerate}
Whenever the second condition holds, $G$ is also a pseudo-Riesz sequence.
If $F$ is a pseudo-Riesz sequence, then $G$ may be chosen so that
\[
    C_GS_F=I_{\ell^2}-P_{\ker S_F}.
\]
\end{proposition}

\begin{proof}
Suppose first that $G$ is Bessel and
\[
    C_GS_F-I_{\ell^2}
\]
has finite rank. Then $C_GS_F$ is Fredholm. Put
\[
    N=\ker(C_GS_F).
\]
Since $C_GS_F$ has closed range, it is bounded below on $N^\perp$. Hence
$S_F$ is bounded below on $N^\perp$, because
\[
    \|C_GS_Fc\|
    \leq
    \|C_G\|\,\|S_Fc\|,
    \qquad c\in N^\perp.
\]
Also, $\ker S_F\subset N$, so $\ker S_F$ is finite-dimensional. Since
\[
    \operatorname{ran}S_F
    =
    S_F(N^\perp)+S_F(N),
\]
the range of $S_F$ is closed. Proposition~\ref{prop:pseudo-riesz-sequence}
now shows that $F$ is a pseudo-Riesz sequence.

Taking adjoints gives
\[
    C_FS_G-I_{\ell^2}
    =
    (C_GS_F-I_{\ell^2})^*.
\]
The same argument, with $F$ and $G$ interchanged, shows that $G$ is a
pseudo-Riesz sequence.

Conversely, suppose that $F$ is a pseudo-Riesz sequence and set
\[
    A:=C_FS_F.
\]
Since $S_F$ has closed range, $A$ has closed range and its Moore--Penrose
inverse $A^\dagger$ is bounded. Define a bounded operator
\[
    S_G:=S_FA^\dagger:\ell^2\longrightarrow\HH,
\]
and let $g_n=S_Ge_n$. Then $G=(g_n)_{n\in\N}$ is Bessel. Since
$A=C_FS_F$ is positive, $A^\dagger$ is self-adjoint, and hence
\[
    C_G=A^\dagger C_F.
\]
Therefore,
\[
    C_GS_F
    =
    A^\dagger C_FS_F
    =
    A^\dagger A
    =
    P_{(\ker S_F)^\perp}
    =
    I_{\ell^2}-P_{\ker S_F}.
\]
Thus $C_GS_F-I_{\ell^2}$ has finite rank.
\end{proof}
\section{Pseudo-Riesz bases}\label{sec:pseudo-riesz-bases}

We now combine the finite-deletion and finite-completion results from the
preceding sections. Recall from the introduction that a pseudo-Riesz basis
is a Bessel sequence that can be transformed into a Riesz basis by deleting
finitely many terms and adjoining finitely many vectors.

\begin{theorem}\label{thm:pseudo-riesz-basis}
Let $F=(f_n)_{n\in\N}$ be a Bessel sequence in $\HH$, with synthesis
operator $S_F$. The following are equivalent:
\begin{enumerate}
    \item $F$ is a pseudo-Riesz basis;
    \item $F$ is a pseudo-Riesz sequence and admits a finite frame
    completion;
    \item $S_F$ is Fredholm.
\end{enumerate}
When these conditions hold, set
\[
    d:=\dim\ker S_F,
    \qquad
    q:=\operatorname{codim}\operatorname{ran}S_F.
\]
There is a finite set $\sigma\subset\N$, with $|\sigma|=d$, and vectors
$h_1,\ldots,h_q\in\HH$ such that
\[
    \bigl(h_1,\ldots,h_q,(f_n)_{n\notin\sigma}\bigr)
\]
is a Riesz basis for $\HH$.

More generally, if $\tau\subset\N$ is finite and
\[
    \bigl(u_1,\ldots,u_m,(f_n)_{n\notin\tau}\bigr)
\]
is a Riesz basis for $\HH$, then
\[
    |\tau|\geq d,
    \qquad
    m=q+|\tau|-d,
    \qquad
    |\tau|-m=\operatorname{ind}S_F.
\]
Consequently, $d$ and $q$ are respectively the least possible numbers of
deleted and adjoined vectors.
\end{theorem}

\begin{proof}
The equivalence of (ii) and (iii) follows from
Proposition~\ref{prop:pseudo-riesz-sequence} and
Proposition~\ref{prop:finite-frame-completion}.

Assume first that $F$ is a pseudo-Riesz basis. Thus, for some finite set
$\tau\subset\N$ and vectors $u_1,\ldots,u_m\in\HH$,
\[
    \bigl(u_1,\ldots,u_m,(f_n)_{n\notin\tau}\bigr)
\]
is a Riesz basis for $\HH$. Let $Q_\tau$ be the orthogonal projection of
$\ell^2$ onto $M_\tau$, and set
\[
    T:=S_FQ_\tau.
\]
The synthesis operator of the displayed Riesz basis is an isomorphism from
$\C^m\oplus M_\tau$ onto $\HH$. Hence
$S_F|_{M_\tau}$ is an isomorphism onto a closed subspace of codimension
$m$. It follows that
\[
    \ker T=\operatorname{span}\{e_j:j\in\tau\},
    \qquad
    \operatorname{codim}\operatorname{ran}T=m.
\]
Thus $T$ is Fredholm and
\[
    \operatorname{ind}T=|\tau|-m.
\]
Since
\[
    S_F-T=S_F(I-Q_\tau)
\]
has finite rank, $S_F$ is Fredholm and
\[
    \operatorname{ind}S_F=|\tau|-m.
\]

Conversely, suppose that $S_F$ is Fredholm. By
Lemma~\ref{lem:coordinate-complement}, there is a finite set
$\sigma\subset\N$, with $|\sigma|=d$, such that
\[
    S_F|_{M_\sigma}:M_\sigma\longrightarrow\operatorname{ran}S_F
\]
is an isomorphism. Hence $(f_n)_{n\notin\sigma}$ is a Riesz basis for
$\operatorname{ran}S_F$. Choose an orthonormal basis
$h_1,\ldots,h_q$ for $(\operatorname{ran}S_F)^\perp$. Then
\[
    \bigl(h_1,\ldots,h_q,(f_n)_{n\notin\sigma}\bigr)
\]
is a Riesz basis for
\[
    (\operatorname{ran}S_F)^\perp\oplus\operatorname{ran}S_F=\HH.
\]
Thus $F$ is a pseudo-Riesz basis.

For the final assertions, let $\tau$ and $u_1,\ldots,u_m$ give any
pseudo-Riesz-basis realization. Since $(f_n)_{n\notin\tau}$ is a Riesz
sequence, $P_\tau$ is injective on $\ker S_F$. Hence
\[
    |\tau|\geq\dim\ker S_F=d.
\]
The first part of the proof gives
\[
    \operatorname{ind}S_F=|\tau|-m.
\]
Since $\operatorname{ind}S_F=d-q$, it follows that
\[
    m=q+|\tau|-d.
\]
The construction above attains $|\sigma|=d$ and $m=q$, which proves the
minimality statements.
\end{proof}

\begin{remark}
The Fredholm index records the imbalance between deletion and adjunction:
\[
    \operatorname{ind}S_F
    =
    \dim\ker S_F-\operatorname{codim}\operatorname{ran}S_F.
\]
It does not, in general, measure the total number of finite modifications.
In particular, index zero does not imply that the sequence is already a
Riesz basis.
\end{remark}

\begin{corollary}\label{cor:near-riesz-basis}
Let $F$ be a Bessel sequence in $\HH$. The following are equivalent:
\begin{enumerate}
    \item $F$ is a near-Riesz basis;
    \item $F$ is a frame and a pseudo-Riesz sequence;
    \item $F$ is a frame and a pseudo-Riesz basis.
\end{enumerate}
\end{corollary}

\begin{proof}
The implication from (i) to (ii) is immediate. If (ii) holds, then $F$
admits a finite frame completion with no added vectors, so
Theorem~\ref{thm:pseudo-riesz-basis} gives (iii). Conversely, if (iii)
holds, then $S_F$ is Fredholm and surjective. The finite-coordinate
construction in the proof of Theorem~\ref{thm:pseudo-riesz-basis} therefore
deletes finitely many terms and leaves a Riesz basis for $\HH$, proving
(i). This is the Fredholm formulation of Holub's characterization
\cite{holub1994bases}.
\end{proof}

\begin{remark}
Let $(h_n)_{n\in\N}$ be a Riesz basis for $\HH$, let
$(\widetilde h_n)_{n\in\N}$ be its biorthogonal Riesz basis, and fix
$m\geq1$. The finite-deficit Riesz sequence
\[
    f_n=h_{m+n},\qquad n\in\N,
\]
is a pseudo-Riesz basis: adjoining $h_1,\ldots,h_m$ gives a Riesz basis
for $\HH$. Nevertheless,
\[
    g_n=\widetilde h_1+\widetilde h_{m+n},\qquad n\in\N,
\]
is biorthogonal to $(f_n)$ but is not a Bessel sequence, since
\[
    \langle f_j,g_n\rangle=\delta_{jn}
    \quad\text{and}\quad
    \sum_{n\in\N}|\langle h_1,g_n\rangle|^2=\infty.
\]
Thus one cannot replace the frame-specific Heil--Yu condition by a
condition on arbitrary biorthogonal sequences in the pseudo-Riesz setting.
This does not preclude an extension based on finite-defect reconstruction
on the whole space, as in Theorem~\ref{thm:automatic-besselness} below.
\end{remark}

\begin{corollary}\label{cor:finite-rank-fredholm-partners}
Let $F$ and $G$ be Bessel sequences in $\HH$. Suppose that either
\[
    S_FC_G-I_{\HH}
\]
or
\[
    C_GS_F-I_{\ell^2}
\]
has finite rank. Then $F$ is a pseudo-Riesz basis if and only if $G$ is a
pseudo-Riesz basis. Whenever these equivalent conditions hold,
\[
    \operatorname{ind}S_F=\operatorname{ind}S_G.
\]
\end{corollary}

\begin{proof}
It is enough to prove the implication from $F$ to $G$, since taking
adjoints gives the corresponding hypothesis with $F$ and $G$ interchanged.
Suppose that $F$ is a pseudo-Riesz basis. By
Theorem~\ref{thm:pseudo-riesz-basis}, $S_F$ is Fredholm. Let
$R:\HH\to\ell^2$ be a Fredholm parametrix for $S_F$.

If
\[
    S_FC_G-I_{\HH}=K
\]
has finite rank, then
\[
    C_G-R=RK-(RS_F-I_{\ell^2})C_G
\]
is compact. If instead
\[
    C_GS_F-I_{\ell^2}=K
\]
has finite rank, then
\[
    C_G-R=KR-C_G(S_FR-I_{\HH})
\]
is compact. In either case, $C_G$ is a compact perturbation of the
Fredholm operator $R$, and is therefore Fredholm. Hence
$S_G=C_G^*$ is Fredholm, so $G$ is a pseudo-Riesz basis.

Finally, the relevant product is a finite-rank perturbation of the identity
and therefore has index zero. Thus
\[
    0=\operatorname{ind}S_F+\operatorname{ind}C_G
      =\operatorname{ind}S_F-\operatorname{ind}S_G,
\]
which proves the index identity.
\end{proof}

\subsection{Automatic Besselness for finite-defect reconstruction}

The preceding corollary assumes that both sequences are Bessel. We now
remove that assumption on the reconstructing sequence, extending the
automatic Besselness phenomenon of Heil and Yu
\cite[Theorem~3.2]{heil2023} to reconstruction with finite-dimensional error.

\begin{definition}\label{def:finite-defect-duals}
Let $F=(f_n)$ be a Bessel sequence in $\HH$. A sequence $G=(g_n)$ is a
\emph{finite-defect alternative dual} of $F$ if there is a bounded
finite-rank operator $K:\HH\to\HH$ such that
\begin{equation}\label{eq:finite-defect-alternative}
    (I_{\HH}-K)x=\sum_{n\in\N}\langle x,g_n\rangle f_n,
    \qquad x\in\HH,
\end{equation}
with convergence in norm. It is a \emph{finite-defect synthesis
pseudo-dual} if, for such an operator $K$,
\begin{equation}\label{eq:finite-defect-synthesis}
    (I_{\HH}-K)x=\sum_{n\in\N}\langle x,f_n\rangle g_n,
    \qquad x\in\HH,
\end{equation}
with convergence in norm. Neither definition assumes that $G$ is Bessel.
\end{definition}

\begin{theorem}\label{thm:automatic-besselness}
Let $F$ be a pseudo-Riesz basis for $\HH$. Every finite-defect alternative
dual and every finite-defect synthesis pseudo-dual $G$ of $F$ is Bessel
and is a pseudo-Riesz basis, with
\[
    \operatorname{ind}S_G=\operatorname{ind}S_F.
\]
The defining reconstruction series converges unconditionally for every
$x\in\HH$.
\end{theorem}

\begin{proof}
Choose a finite set $\sigma$ so that $(f_n)_{n\notin\sigma}$ is a Riesz
basis for $M=\operatorname{ran}S_F$, and let $A>0$ be its lower Riesz
bound. Such a choice is possible by
Lemma~\ref{lem:coordinate-complement}.

Suppose first that \eqref{eq:finite-defect-alternative} holds. Applying
the lower Riesz bound to finite partial sums of the series with the
indices in $\sigma$ removed, and then passing to the limit, gives
\[
\begin{split}
    A\sum_{n\notin\sigma}|\langle x,g_n\rangle|^2
    &\leq
    \left\|(I_{\HH}-K)x
       -\sum_{j\in\sigma}\langle x,g_j\rangle f_j\right\|^2\\
    &\leq
    \left(\|I_{\HH}-K\|
       +\sum_{j\in\sigma}\|g_j\|\,\|f_j\|\right)^2\|x\|^2.
\end{split}
\]
Adding the finitely many terms indexed by $\sigma$ proves that $G$ is
Bessel.

Suppose next that \eqref{eq:finite-defect-synthesis} holds. Let
$(u_n)_{n\notin\sigma}\subset M$ be the biorthogonal Riesz basis to
$(f_n)_{n\notin\sigma}$. Substitution of $x=u_k$, $k\notin\sigma$,
into that identity gives
\[
    g_k=(I_{\HH}-K)u_k
       -\sum_{j\in\sigma}\langle u_k,f_j\rangle g_j.
\]
The sequence $((I_{\HH}-K)u_k)_{k\notin\sigma}$ is Bessel. Fix
$j\in\sigma$. The coefficients
$(\langle u_k,f_j\rangle)_{k\notin\sigma}$ belong to $\ell^2$, since
$(u_k)$ is Bessel in $\HH$. Hence each sequence
$(\langle u_k,f_j\rangle g_j)_{k\notin\sigma}$ is Bessel. Their finite
sum, followed by the addition of the finitely many $g_j$ with
$j\in\sigma$, proves that $G$ is Bessel.

In the first case, $S_FC_G=I_{\HH}-K$. In the second case,
$S_GC_F=I_{\HH}-K$, so $S_FC_G=I_{\HH}-K^*$. Corollary~
\ref{cor:finite-rank-fredholm-partners} now gives the pseudo-Riesz
property and the index identity. Finally, the analysis coefficients are
in $\ell^2$ and the synthesizing sequence is Bessel, so each defining
series converges unconditionally.
\end{proof}

\begin{corollary}
\label{cor:finite-defect-heil-yu}
Let $\HH$ be infinite-dimensional, and let $F$ be a Bessel sequence in
$\HH$ admitting a finite frame completion. The following are equivalent:
\begin{enumerate}
    \item $F$ is a pseudo-Riesz basis;
    \item every finite-defect alternative dual of $F$ is Bessel;
    \item every finite-defect synthesis pseudo-dual of $F$ is Bessel.
\end{enumerate}
If $F$ has at most finitely many zero terms, these conditions are also
equivalent to each of the following:
\begin{enumerate}
    \item[(a)] for every finite-defect alternative dual $G=(g_n)$ of
    $F$ and every $x\in\HH$, the series
    $\sum_n\langle x,g_n\rangle f_n$ converges unconditionally;
    \item[(b)] for every finite-defect synthesis pseudo-dual $G=(g_n)$
    of $F$ and every $x\in\HH$, the series
    $\sum_n\langle x,f_n\rangle g_n$ converges unconditionally.
\end{enumerate}
The two classes of finite-defect duals are nonempty. When these conditions
hold, every such dual is a pseudo-Riesz basis with the same synthesis
index as $F$.
\end{corollary}

\begin{proof}
By Proposition~\ref{prop:finite-frame-completion},
$M=\operatorname{ran}S_F$ is closed and has finite codimension, and $F$
is a frame for $M$. Its canonical dual in $M$, viewed as a sequence in
$\HH$, gives both reconstruction identities with $I_{\HH}-K=P_M$.
This proves nonemptiness. Theorem~\ref{thm:automatic-besselness} proves
that (i) implies (ii) and (iii), as well as the final assertion.

Assume (ii). Every alternative dual of $F$ as a frame for $M$, with its
vectors in $M$, satisfies \eqref{eq:finite-defect-alternative} on $\HH$
with $K=P_{M^\perp}$. It is therefore Bessel, and hence is a frame for
$M$. Theorem~\ref{thm:heil-yu}, applied in $M$, shows that $F$ is a
near-Riesz basis for $M$. Since $M$ has finite codimension in $\HH$,
$F$ is a pseudo-Riesz basis for $\HH$. The same argument using synthesis
pseudo-duals proves that (iii) implies (i).

Suppose now that $F$ has at most finitely many zero terms.
Theorem~\ref{thm:automatic-besselness} shows that (i) implies both (a)
and (b). Conversely, under (a), every alternative dual of $F$ as a
frame for $M$, viewed as a sequence in $\HH$, is a finite-defect
alternative dual with error $K=P_{M^\perp}$. Its reconstruction series
is therefore unconditional for every $x\in M$. The unconditional
convergence characterization in Theorem~\ref{thm:heil-yu} implies that
$F$ is a near-Riesz basis for $M$, hence a pseudo-Riesz basis for
$\HH$. The same argument with synthesis pseudo-duals proves that (b)
implies (i).
\end{proof}

\begin{remark}
The requirement that the identities hold on all of $\HH$ cannot be
replaced by reconstruction only on $M$. For an orthonormal basis $(e_n)$,
the sequences $f_n=e_{n+1}$ and $g_n=e_{n+1}+e_1$ satisfy
\[
    x=\sum_n\langle x,g_n\rangle f_n,
    \qquad x\in\overline{\operatorname{span}}\{e_n:n\geq2\},
\]
but $G$ is not Bessel. For $x=e_1$ the same series is
$\sum_n e_{n+1}$, which does not converge. Thus biorthogonality or
reconstruction on the closed span alone does not contradict
Theorem~\ref{thm:automatic-besselness}.
\end{remark}

\section{Perturbation results}\label{sec:perturbations}

We study perturbations at the level of synthesis operators. Pseudo-Riesz
sequences, Bessel sequences admitting a finite frame completion, and
pseudo-Riesz bases correspond respectively to left semi-Fredholm, right
semi-Fredholm, and Fredholm synthesis operators.
The frame case of the results below is closely related to perturbation
results for near-Riesz bases; see \cite{casazza1998frames}.

Write $c_{00}$ for the finitely supported sequences in $\ell^2$. If $T$ has
closed range, its reduced minimum modulus is
\[
    \gamma(T)
    :=
    \inf\bigl\{\|Tc\|:c\in(\ker T)^\perp,\ \|c\|=1\bigr\}>0.
\]

\subsection{Relative perturbations}

We use the following relative perturbation theorem of Kato.

\begin{theorem}[Kato's first stability theorem
{\cite[Theorem~IV.5.22]{Kato1980:chIV}}]
\label{thm:kato-relative-perturbation}
Let $T\in\mathcal{B}(\HH,\KK)$ be semi-Fredholm, and let
$A\in\mathcal{B}(\HH,\KK)$ satisfy
\[
    \|Ax\|\leq a\|x\|+b\|Tx\|,
    \qquad x\in\HH,
\]
where $a,b\geq0$, $b<1$, and
\[
    a<(1-b)\gamma(T).
\]
Then $T+A$ is semi-Fredholm. Moreover,
\[
    \dim\ker(T+A)\leq\dim\ker T,
    \qquad
    \operatorname{codim}\operatorname{ran}(T+A)
    \leq\operatorname{codim}\operatorname{ran}T.
\]
If $T$ is Fredholm, then $T+A$ is Fredholm and
\[
    \operatorname{ind}(T+A)=\operatorname{ind}T.
\]
\end{theorem}

\begin{theorem}\label{thm:relative-perturbation}
Let $F=(f_n)_{n\in\N}$ be a Bessel sequence in $\HH$ for which $S_F$ is
semi-Fredholm, and let $G=(g_n)_{n\in\N}$ be a sequence in $\HH$. Suppose
that, for every $c=(c_n)\in c_{00}$,
\begin{equation}\label{eq:relative-perturbation}
    \left\|\sum_{n\in\N}c_n(g_n-f_n)\right\|
    \leq
    \lambda\|S_Fc\|+\mu\|c\|,
\end{equation}
where
\[
    0\leq\lambda<1,
    \qquad
    0\leq\mu<(1-\lambda)\gamma(S_F).
\]
Then $G$ is a Bessel sequence. Moreover:
\begin{enumerate}
    \item If $F$ is a pseudo-Riesz sequence, then $G$ is a pseudo-Riesz
    sequence and
    \[
        \dim\ker S_G\leq\dim\ker S_F.
    \]
    \item If $F$ admits a finite frame completion, then $G$ admits a finite
    frame completion and
    \[
        \operatorname{codim}\operatorname{ran}S_G
        \leq
        \operatorname{codim}\operatorname{ran}S_F.
    \]
    \item If $F$ is a pseudo-Riesz basis, then $G$ is a pseudo-Riesz basis
    and
    \[
        \operatorname{ind}S_G=\operatorname{ind}S_F.
    \]
\end{enumerate}
\end{theorem}

\begin{proof}
For $c\in c_{00}$, define
\[
    D_0c:=\sum_{n\in\N}c_n(g_n-f_n).
\]
By \eqref{eq:relative-perturbation},
\[
    \|D_0c\|
    \leq
    \bigl(\lambda\|S_F\|+\mu\bigr)\|c\|.
\]
Thus $D_0$ extends to a bounded operator
\[
    D:\ell^2\longrightarrow\HH.
\]
The bounded operator $S_F+D$ is the synthesis operator of $G$, so $G$ is
Bessel and
\[
    S_G=S_F+D.
\]
The estimate \eqref{eq:relative-perturbation} says that $D$ is
$S_F$-bounded with Kato constants $a=\mu$ and $b=\lambda$. Theorem~
\ref{thm:kato-relative-perturbation}, together with
Proposition~\ref{prop:pseudo-riesz-sequence},
Proposition~\ref{prop:finite-frame-completion}, and
Theorem~\ref{thm:pseudo-riesz-basis}, yields the three conclusions.
\end{proof}

\subsection{Hilbert--Schmidt perturbations}

Square-summable perturbations are often called quadratic or Bari-type
perturbations. In the present setting they yield a Hilbert--Schmidt
difference of synthesis operators, so no smallness condition is needed.

\begin{theorem}\label{thm:quadratic-perturbation}
Let $F=(f_n)_{n\in\N}$ be a Bessel sequence in $\HH$, and let
$G=(g_n)_{n\in\N}$ be a sequence in $\HH$ such that
\[
    \delta^2:=\sum_{n\in\N}\|f_n-g_n\|^2<\infty.
\]
Then $G$ is a Bessel sequence and $S_G-S_F$ is Hilbert--Schmidt. Moreover:
\begin{enumerate}
    \item If $F$ is a pseudo-Riesz sequence, then $G$ is a pseudo-Riesz
    sequence.
    \item If $F$ admits a finite frame completion, then $G$ admits a finite
    frame completion.
    \item If $F$ is a pseudo-Riesz basis, then $G$ is a pseudo-Riesz basis
    and
    \[
        \operatorname{ind}S_G=\operatorname{ind}S_F.
    \]
\end{enumerate}
\end{theorem}

\begin{proof}
For $c\in c_{00}$,
\[
    \left\|\sum_{n\in\N}c_n(g_n-f_n)\right\|
    \leq
    \left(\sum_{n\in\N}|c_n|^2\right)^{1/2}
    \left(\sum_{n\in\N}\|g_n-f_n\|^2\right)^{1/2}
    =
    \delta\|c\|.
\]
Hence the difference map extends to a bounded operator
\[
    D:\ell^2\longrightarrow\HH
\]
with
\[
    \sum_{n\in\N}\|De_n\|^2
    =
    \sum_{n\in\N}\|g_n-f_n\|^2
    =
    \delta^2.
\]
Thus $D$ is Hilbert--Schmidt, hence compact. The operator $S_F+D$ maps
$e_n$ to $g_n$ for every $n$, so it is the synthesis operator $S_G$ of a
Bessel sequence $G$. The
standard compact-perturbation theorem for semi-Fredholm operators
\cite[Theorem~IV.5.26]{Kato1980:chIV} preserves the left and right
semi-Fredholm properties; in the Fredholm case it also preserves the
index. The conclusion follows from the characterizations established in
Sections~3--5.
\end{proof}

\section{Non-harmonic exponentials and the Hru\v{s}\v{c}ev--Nikol'skii--Pavlov model}
\label{sec:nonharmonic-exponentials}

Let $\Lambda=(\lambda_n)_{n\in\N}$ be a sequence of distinct real
numbers, and set
\[
    e_\lambda(t):=e^{2\pi i\lambda t},
    \qquad 0<t<1,
    \qquad
    E_\Lambda=(e_{\lambda_n})_{n\in\N}.
\]
When set operations are used below, $\Lambda$ also denotes the underlying
set $\{\lambda_n:n\in\N\}$.
We now give the concrete upper-half-plane model used by
Hru\v{s}\v{c}ev, Nikol'skii, and Pavlov.  The reflection and vertical
shift below are essential: a Blaschke product has its zeros in
$\C_+$, not on the real line.

Fix $\eta>0$, and define
\[
    \zeta_n:=-2\pi\lambda_n+i\eta\in\C_+,
    \qquad
    \Theta(z):=e^{iz}.
\]
For an inner function $\Phi$ on $\C_+$, write
\[
    \KK_\Phi:=\HH^2\ominus\Phi\HH^2
\]
and denote by $P_\Phi$ the orthogonal projection of $\HH^2$ onto
$\KK_\Phi$.  With the standard upper-half-plane normalization, the
reproducing kernel of $\KK_\Phi$ at $w\in\C_+$ is
\[
    k_w(z):=\frac{i}{z-\overline w},
    \qquad
    k^\Phi_w(z)
    :=
    P_\Phi k_w(z)
    =
    \frac{i\bigl(1-\overline{\Phi(w)}\Phi(z)\bigr)}
         {z-\overline w}.
\]
The Fourier--Laplace transform identifies $L^2(0,1)$ with
$\KK_\Theta$.  After composing it with the invertible multiplier
$M_\eta f(t)=e^{-\eta t}f(t)$, we obtain a bounded isomorphism
\[
    U_\eta:L^2(0,1)\longrightarrow\KK_\Theta,
    \qquad
    (U_\eta f)(z)
    =
    \int_0^1 e^{izt}e^{-\eta t}f(t)\,dt .
\]
Indeed, for every $n$,
\begin{equation}\label{eq:hnp-kernel-image}
\begin{split}
    (U_\eta e_{\lambda_n})(z)
    &=
    \int_0^1e^{i(z-\overline{\zeta_n})t}\,dt\\
    &=
    \frac{i\bigl(1-e^{i(z-\overline{\zeta_n})}\bigr)}
         {z-\overline{\zeta_n}}
    =
    k^\Theta_{\zeta_n}(z).
\end{split}
\end{equation}
This is the Fourier--Laplace realization underlying the
Hru\v{s}\v{c}ev--Nikol'skii--Pavlov theory
\cite[Part~I, \S\S1--4]{hruscev1981}.

The Carleson (interpolating) condition on the lifted sequence is
\begin{equation}\label{eq:hnp-carleson}
    \delta_\Lambda
    :=
    \inf_n\prod_{m\ne n}
    \left|
        \frac{\zeta_n-\zeta_m}{\zeta_n-\overline{\zeta_m}}
    \right|
    >0.
\end{equation}
For points on this horizontal line it is equivalent to uniform separation
of the real frequencies. The following observation also shows that it is
automatic for the pseudo-Riesz systems considered here.

\begin{lemma}
\label{lem:horizontal-carleson}
Let $\Lambda=(\lambda_n)$ consist of distinct real frequencies.
\begin{enumerate}
    \item For each $\eta>0$, condition~\eqref{eq:hnp-carleson} holds
    if and only if $\Lambda$ is uniformly separated.
    \item If $E_\Lambda$ is a pseudo-Riesz sequence, then $\Lambda$ is
    uniformly separated. In particular, every pseudo-Riesz basis of
    distinct real exponentials satisfies \eqref{eq:hnp-carleson} for
    every $\eta>0$.
\end{enumerate}
\end{lemma}

\begin{proof}
Suppose first that $|\lambda_n-\lambda_m|\geq\delta>0$ for $n\ne m$.
The $k$th point to either side of a fixed $\lambda_n$, when it exists,
is at distance at least $k\delta$. Since
\[
    \left|\frac{\zeta_n-\zeta_m}
                    {\zeta_n-\overline{\zeta_m}}\right|
    =\left(1+\frac{\eta^2}
        {\pi^2|\lambda_n-\lambda_m|^2}\right)^{-1/2},
\]
we obtain, uniformly in $n$,
\[
    \prod_{m\ne n}
    \left|\frac{\zeta_n-\zeta_m}
                    {\zeta_n-\overline{\zeta_m}}\right|
    \geq
    \prod_{k=1}^{\infty}
    \left(1+\frac{\eta^2}{\pi^2k^2\delta^2}\right)^{-1}>0.
\]
The last product is positive because the sum of its logarithms converges.
Uniform separation also implies
$\sum_n\operatorname{Im}\zeta_n/(1+|\zeta_n|^2)<\infty$, so the
associated Blaschke product exists. Conversely, every factor in a
Carleson product is at most one. A positive uniform lower bound for the
products therefore bounds each factor away from zero, which gives a
positive lower bound for $|\lambda_n-\lambda_m|$.

For (ii), delete finitely many terms so that the remaining exponentials
form a Riesz sequence with lower bound $A>0$. For distinct remaining
frequencies $\lambda,\mu$,
\[
\begin{split}
    2A
    &\leq\|e_\lambda-e_\mu\|_{L^2(0,1)}^2\\
    &\leq 4\pi^2|\lambda-\mu|^2\int_0^1t^2\,dt
     =\frac{4\pi^2}{3}|\lambda-\mu|^2.
\end{split}
\]
Thus the remaining frequency set is uniformly separated and has no
finite accumulation points. Each deleted frequency is distinct from that
set and has positive distance from it. Restoring finitely many such
points preserves uniform separation. Part (i) finishes the proof.
\end{proof}

For the operator correspondences below we assume
\eqref{eq:hnp-carleson}; by the lemma this imposes no extra restriction
on pseudo-Riesz sequences or pseudo-Riesz bases of distinct real
exponentials.
Let $B$ be the Blaschke product whose simple zeros are the points
$\zeta_n$.  Condition~\eqref{eq:hnp-carleson} implies that the
kernels $(k^B_{\zeta_n})_{n\in\N}$ form a Riesz basis for $\KK_B$.
Here the usual normalized-kernel statement has been written without
normalization because all the points $\zeta_n$ have the same imaginary
part.
Since $B(\zeta_n)=0$, these kernels are the ordinary Hardy kernels
\[
    k^B_{\zeta_n}(z)
    =
    \frac{i}{z-\overline{\zeta_n}}.
\]
Consequently, the map
\[
    J:\ell^2\longrightarrow\KK_B,
    \qquad
    Je_n=k^B_{\zeta_n},
\]
is an isomorphism.  Notice also that
$|\Theta(\zeta_n)|=e^{-\eta}<1$, which is the comparison condition
in the Hru\v{s}\v{c}ev--Nikol'skii--Pavlov criterion.
Initially, $S_{E_\Lambda}$ is understood on the finitely supported
sequences; Proposition~\ref{prop:hnp-intertwining} shows that it extends
to a bounded synthesis operator on $\ell^2$.

\begin{proposition}\label{prop:hnp-intertwining}
Let
\[
    Q:=P_\Theta\big|_{\KK_B}:\KK_B\longrightarrow\KK_\Theta.
\]
Under condition~\eqref{eq:hnp-carleson},
\begin{equation}\label{eq:hnp-model-factorization}
    U_\eta S_{E_\Lambda}=QJ.
\end{equation}
\end{proposition}

\begin{proof}
For every $n$,
\[
    QJe_n
    =
    P_\Theta k^B_{\zeta_n}
    =
    P_\Theta k_{\zeta_n}
    =
    k^\Theta_{\zeta_n}.
\]
The last expression equals $U_\eta e_{\lambda_n}$ by
\eqref{eq:hnp-kernel-image}. The bounded operator $U_\eta^{-1}QJ$
therefore agrees with exponential synthesis on the finitely supported
sequences and supplies its bounded extension to $\ell^2$. This proves
\eqref{eq:hnp-model-factorization} on all of $\ell^2$.
\end{proof}

\begin{theorem}
\label{thm:exponential-model-transfer}
Under condition~\eqref{eq:hnp-carleson}, the following statements hold:
\begin{enumerate}
    \item $E_\Lambda$ is a pseudo-Riesz sequence if and only if $Q$ is
    left semi-Fredholm.
    \item $E_\Lambda$ admits a finite frame completion if and only if
    $Q$ is right semi-Fredholm.
    \item $E_\Lambda$ is a pseudo-Riesz basis if and only if $Q$ is
    Fredholm.
\end{enumerate}
When $Q$ is Fredholm, the least number of deletions and the least
number of added vectors needed to obtain a Riesz basis are, respectively,
\[
    \dim\ker Q
    \quad\text{and}\quad
    \dim\ker Q^*.
\]
\end{theorem}

\begin{proof}
Proposition~\ref{prop:hnp-intertwining} gives
\[
    S_{E_\Lambda}=U_\eta^{-1}QJ.
\]
The maps $U_\eta$ and $J$ are isomorphisms.  Thus, composition on either
side by them preserves closed range, the dimensions of the kernel and
cokernel, and all left, right, and two-sided semi-Fredholm properties.
The assertions and the minimality statements now follow from
Proposition~\ref{prop:pseudo-riesz-sequence},
Proposition~\ref{prop:finite-frame-completion}, and
Theorem~\ref{thm:pseudo-riesz-basis}.
\end{proof}

We next pass from the projection $Q$ to a Toeplitz operator.  Let $P_+$
denote the orthogonal projection of $L^2(\R)$ onto $\HH^2$, and for
$\varphi\in L^\infty(\R)$ set
\[
    T_\varphi f:=P_+(\varphi f),
    \qquad f\in\HH^2.
\]

\begin{proposition}\label{prop:model-projection-toeplitz}
The operator
\[
    Q=P_\Theta\big|_{\KK_B}:\KK_B\longrightarrow\KK_\Theta
\]
is left semi-Fredholm, right semi-Fredholm, or Fredholm if and only if
$T_{\overline{B}\Theta}$ has the corresponding property.  In the
Fredholm case,
\[
    \operatorname{ind}Q
    =
    \operatorname{ind}T_{\overline{B}\Theta}.
\]
\end{proposition}

\begin{proof}
For an inner function $\Phi$, let
\[
    C_\Phi f:=\Phi\overline f,
    \qquad f\in\KK_\Phi,
\]
be the standard antiunitary involution of $\KK_\Phi$.  Also let
\[
    \mathcal V_\Phi:\KK_\Phi\oplus\HH^2\longrightarrow\HH^2,
    \qquad
    \mathcal V_\Phi(h,g):=h+\Phi g.
\]
This is a unitary map.

For $h\in\KK_B$, decompose $h=P_\Theta h+\Theta g$ with $g\in\HH^2$.
Then
\[
    P_+(\Theta\overline h)
    =
    \Theta\overline{P_\Theta h}
    =
    C_\Theta P_\Theta h.
\]
It follows that
\[
    T_{\overline{B}\Theta}(C_Bh)
    =
    P_+(\overline B\Theta B\overline h)
    =
    C_\Theta Qh.
\]
Since $C_B$ is involutive, this gives
\[
    T_{\overline{B}\Theta}x=C_\Theta QC_Bx,
    \qquad x\in\KK_B.
\]
Moreover,
\[
    T_{\overline{B}\Theta}(Bg)=\Theta g,
    \qquad g\in\HH^2.
\]
Hence
\[
    \mathcal V_\Theta^*T_{\overline{B}\Theta}\mathcal V_B
    =
    \bigl(C_\Theta QC_B\bigr)\oplus I_{\HH^2}.
\]
The first block has the same closed-range property, kernel dimension, and
cokernel dimension as $Q$.  The assertions and the index identity follow.
\end{proof}

\begin{corollary}[Hru\v{s}\v{c}ev--Nikol'skii--Pavlov]
\label{cor:hnp-riesz-basis-case}\label{(5.1.1)}
Under condition~\eqref{eq:hnp-carleson}, the following are equivalent:
\begin{enumerate}
    \item $E_\Lambda$ is a Riesz basis for $L^2(0,1)$;
    \item $Q=P_\Theta|_{\KK_B}$ is an isomorphism;
    \item $T_{\overline{B}\Theta}$ is invertible on $\HH^2$.
\end{enumerate}
\end{corollary}

\begin{proof}
The equivalence of the first two assertions follows from
\eqref{eq:hnp-model-factorization}; the equivalence of the last two
follows from Proposition~\ref{prop:model-projection-toeplitz}.  This is
the classical Hru\v{s}\v{c}ev--Nikol'skii--Pavlov criterion obtained by
combining their Theorems~2/2$'$ with Theorem~4
\cite[Part~I, \S\S2--3]{hruscev1981}.
\end{proof}

\begin{corollary}\label{cor:nonharmonic-toeplitz}\label{(5.8.3)}
Under condition~\eqref{eq:hnp-carleson},
\begin{align*}
    E_\Lambda\text{ is a pseudo-Riesz sequence}
    &\quad\Longleftrightarrow\quad
    T_{\overline{B}\Theta}\text{ is left semi-Fredholm},\\
    E_\Lambda\text{ admits a finite frame completion}
    &\quad\Longleftrightarrow\quad
    T_{\overline{B}\Theta}\text{ is right semi-Fredholm},\\
    E_\Lambda\text{ is a pseudo-Riesz basis}
    &\quad\Longleftrightarrow\quad
    T_{\overline{B}\Theta}\text{ is Fredholm}.
\end{align*}
If these equivalent Fredholm conditions hold, then
\[
    \dim\ker S_{E_\Lambda}
    =
    \dim\ker T_{\overline{B}\Theta},
    \qquad
    \dim\ker S_{E_\Lambda}^*
    =
    \dim\ker T_{\overline{B}\Theta}^*,
\]
and these two dimensions are the minimal deletion and completion numbers,
respectively.
\end{corollary}

\begin{proof}
Combine Theorem~\ref{thm:exponential-model-transfer} with
Proposition~\ref{prop:model-projection-toeplitz}.
\end{proof}

\section{Toeplitz indices and finite modifications}
\label{sec:toeplitz-factors}

The concrete model above has a useful additional feature: adding or
deleting a real frequency changes $B$ by a finite Blaschke factor.  Thus
the finite modifications in this section remain within the same
Hru\v{s}\v{c}ev--Nikol'skii--Pavlov framework.
Related finite-replacement phenomena for kernel systems are discussed in
\cite[Part~II, \S3, Lemma~3.4]{hruscev1981}.

\begin{lemma}\label{lem:finite-carleson-stability}
The Carleson condition~\eqref{eq:hnp-carleson} is preserved after deleting
finitely many points of $(\zeta_n)$ or adjoining finitely many distinct
points of the form
\[
    \zeta_\gamma:=-2\pi\gamma+i\eta,
    \qquad
    \gamma\in\R\setminus\{\lambda_n:n\in\N\}.
\]
\end{lemma}

\begin{proof}
Deleting points increases the products in \eqref{eq:hnp-carleson}.
For additions, the original frequency set is uniformly separated by
Lemma~\ref{lem:horizontal-carleson} and hence
has no finite accumulation points. Each distinct new real frequency has
positive distance from it. Adjoining finitely many such frequencies
preserves uniform separation, so another application of the lemma gives
\eqref{eq:hnp-carleson} for the enlarged set.
\end{proof}

\begin{proposition}\label{prop:toeplitz-finite-factor}
Let $\varphi\in L^\infty(\R)$, and suppose that $T_\varphi$ is Fredholm on
$\HH^2$.  If $B_1$ and $B_2$ are finite Blaschke products of degrees
$d_1$ and $d_2$, respectively, then $T_{\varphi B_1\overline{B_2}}$ is
Fredholm and
\[
    \operatorname{ind}T_{\varphi B_1\overline{B_2}}
    =
    \operatorname{ind}T_\varphi-d_1+d_2.
\]
\end{proposition}

\begin{proof}
For $F\in\HH^\infty$, the Toeplitz product identities give
\[
    T_{\overline F}T_\varphi=T_{\overline F\varphi},
    \qquad
    T_\varphi T_F=T_{\varphi F}.
\]
Therefore,
\[
    T_{\varphi B_1\overline{B_2}}
    =
    T_{\overline{B_2}}T_\varphi T_{B_1}.
\]
The operator $T_{B_j}$ is multiplication by $B_j$; it has index $-d_j$
because its range is $B_j\HH^2$ of codimension $d_j$.  Its adjoint
$T_{\overline{B_j}}$ has index $d_j$.  The conclusion follows from
additivity of the Fredholm index.
\end{proof}

\begin{proposition}\label{prop:zero-index-toeplitz}
Let $0\ne\varphi\in L^\infty(\R)$.  If $T_\varphi$ is Fredholm, then
\[
    T_\varphi\text{ is invertible}
    \quad\Longleftrightarrow\quad
    \operatorname{ind}T_\varphi=0.
\]
\end{proposition}

\begin{proof}
Coburn's alternative, transferred to the upper half-plane by the Cayley
transform, says that either $\ker T_\varphi=\{0\}$ or
$\ker T_\varphi^*=\{0\}$; see, for example, \cite{BottSilber2}.  If the
Fredholm index is zero, the two finite dimensions are equal, so both
kernels are trivial.  Thus $T_\varphi$ is invertible.  The converse is
immediate.
\end{proof}

\begin{theorem}\label{thm:arbitrary-frequency-correction}
Assume condition~\eqref{eq:hnp-carleson}, and put
\[
    \varphi:=\overline B\Theta.
\]
Suppose that $T_\varphi$ is Fredholm.
\begin{enumerate}
    \item If $\operatorname{ind}T_\varphi=-q<0$, then for every set
    $\Gamma\subset\R\setminus\{\lambda_n:n\in\N\}$ of $q$ distinct
    frequencies, the system
    \[
        E_{\Lambda\cup\Gamma}
        =
        (e_{\lambda_n})_{n\in\N}\cup(e_\gamma)_{\gamma\in\Gamma}
    \]
    is a Riesz basis for $L^2(0,1)$.
    \item If $\operatorname{ind}T_\varphi=p>0$, then for every
    $\sigma\subset\N$ with $|\sigma|=p$, the system
    \[
        (e_{\lambda_n})_{n\notin\sigma}
    \]
    is a Riesz basis for $L^2(0,1)$.
    \item If $\operatorname{ind}T_\varphi=0$, then $E_\Lambda$ is already
    a Riesz basis for $L^2(0,1)$.
\end{enumerate}
\end{theorem}

\begin{proof}
For a finite set $\Gamma$ as in the first assertion, let $B_\Gamma$ be
the finite Blaschke product with zeros $\zeta_\gamma$,
$\gamma\in\Gamma$.  The Blaschke product of the enlarged sequence can be
chosen as $BB_\Gamma$.  Lemma~\ref{lem:finite-carleson-stability} permits
the use of Corollary~\ref{cor:hnp-riesz-basis-case} for this enlarged
system.  Its Toeplitz symbol is
\[
    \overline{BB_\Gamma}\Theta
    =
    \overline{B_\Gamma}\varphi.
\]
Proposition~\ref{prop:toeplitz-finite-factor} gives
\[
    \operatorname{ind}T_{\overline{B_\Gamma}\varphi}
    =
    \operatorname{ind}T_\varphi+|\Gamma|
    =
    0.
\]
The Toeplitz operator is therefore invertible by
Proposition~\ref{prop:zero-index-toeplitz}, proving the first assertion.

For the second assertion, let $B_\sigma$ be the finite Blaschke product
with zeros $\zeta_n$, $n\in\sigma$, and choose $B^\sigma$ so that
$B=B_\sigma B^\sigma$.  The remaining system has Blaschke product
$B^\sigma$ and Toeplitz symbol
\[
    \overline{B^\sigma}\Theta
    =
    \varphi B_\sigma.
\]
Its Toeplitz operator has index
\[
    \operatorname{ind}T_{\varphi B_\sigma}
    =
    \operatorname{ind}T_\varphi-|\sigma|
    =
    0.
\]
Again, Propositions~\ref{prop:toeplitz-finite-factor} and
\ref{prop:zero-index-toeplitz}, followed by
Corollary~\ref{cor:hnp-riesz-basis-case}, prove the assertion.  The last
assertion follows directly from
Proposition~\ref{prop:zero-index-toeplitz} and the same corollary.
\end{proof}

For distinct real frequencies, the Carleson condition required by the
preceding model-space argument is automatic under the pseudo-Riesz
hypothesis, by Lemma~\ref{lem:horizontal-carleson}. Consequently, the
classification below requires no additional separation or Carleson
assumption.

\begin{theorem}
\label{thm:real-exponential-classification}
Let $\Lambda=(\lambda_n)_{n\in\N}$ consist of distinct real frequencies,
and suppose that $E_\Lambda$ is a pseudo-Riesz basis for $L^2(0,1)$.
Put $k=\operatorname{ind}S_{E_\Lambda}$. Exactly one of the following
alternatives holds:
\begin{enumerate}
    \item If $k=p>0$, then $E_\Lambda$ is a near-Riesz basis with
    excess $p$, and deleting any $p$ original terms gives a Riesz basis.
    \item If $k=-q<0$, then $E_\Lambda$ is a Riesz sequence with
    deficit $q$, and adjoining any $q$ distinct real-frequency
    exponentials with frequencies outside $\Lambda$ gives a Riesz basis.
    \item If $k=0$, then $E_\Lambda$ is already a Riesz basis.
\end{enumerate}
In particular, $\ker S_{E_\Lambda}$ and $\ker S_{E_\Lambda}^*$ cannot
both be nonzero. The indicated numbers of corrections are minimal, and
no separate separation or Carleson assumption is needed.
\end{theorem}

\begin{proof}
Lemma~\ref{lem:horizontal-carleson} shows that $\Lambda$ is uniformly
separated and satisfies \eqref{eq:hnp-carleson}. Corollary~
\ref{cor:nonharmonic-toeplitz} and Proposition~
\ref{prop:model-projection-toeplitz} identify the kernel and cokernel
dimensions and the index of $S_{E_\Lambda}$ with those of the Fredholm
operator $T_{\overline B\Theta}$.

Coburn's alternative shows that this Fredholm scalar Toeplitz operator
cannot have both a nontrivial kernel and a nontrivial cokernel. Thus, if
$k=p>0$, the synthesis operator is surjective with $p$-dimensional
kernel, so $E_\Lambda$ is a near-Riesz basis with excess $p$. If
$k=-q<0$, the synthesis operator is injective with closed range of
codimension $q$, so $E_\Lambda$ is a Riesz sequence with deficit $q$.
If $k=0$, the synthesis operator is invertible, so $E_\Lambda$ is a
Riesz basis. The asserted corrections follow from
Theorem~\ref{thm:arbitrary-frequency-correction}, and their minimality
follows from Theorem~\ref{thm:pseudo-riesz-basis}.
\end{proof}

\begin{remark}
The definition permits deletions and adjunctions of arbitrary vectors in
the ambient Hilbert space. Theorem~\ref{thm:real-exponential-classification}
shows that
an optimal correction can always be made using only deletions or only
adjunctions of exponentials, and that every choice of the required number
works. This rigidity contrasts with general pseudo-Riesz bases, which
may have simultaneous nonzero redundancy and deficiency. It does not
extend to arbitrary exponential frames or Riesz sequences with infinite
defect.
\end{remark}

As a consequence of Kadec's $1/4$-theorem and the classification above,
the Kadec bound need only hold outside a finite set of indices, provided
the frequencies remain distinct.

\begin{corollary}
\label{cor:asymptotic-kadec}
Let $(\lambda_n)_{n\in\Z}$ be a sequence of distinct real numbers.
If, for some $a\in\R$,
\[
    \limsup_{|n|\to\infty}|\lambda_n-n-a|<\frac14,
\]
then $(e_{\lambda_n})_{n\in\Z}$ is a Riesz basis for $L^2(0,1)$.
In particular, the conclusion holds if $\lambda_n=n+a+o(1)$ as
$|n|\to\infty$.
\end{corollary}

\begin{proof}
Choose $0<\delta<1/4$ and $N\in\N$ such that
$|\lambda_n-n-a|\leq\delta$ for $|n|>N$, and define
\[
    \mu_n=
    \begin{cases}
        n+a, & |n|\leq N,\\
        \lambda_n, & |n|>N.
    \end{cases}
\]
Kadec's $1/4$-theorem \cite{young1980}, followed by the unitary
modulation $f(t)\mapsto e^{2\pi iat}f(t)$, shows that
$(e_{\mu_n})_{n\in\Z}$ is a Riesz basis. The synthesis operator of
$(e_{\lambda_n})_{n\in\Z}$ differs from that of
$(e_{\mu_n})_{n\in\Z}$ by the bounded finite-rank operator
\[
    c\longmapsto\sum_{|n|\leq N}c_n(e_{\lambda_n}-e_{\mu_n}).
\]
It is therefore Fredholm of index zero. Since the original frequencies
are distinct and real, Theorem~\ref{thm:real-exponential-classification}
implies that their exponential system is a Riesz basis. The change from
$\N$ to $\Z$ as index set is immaterial: a bijection induces a unitary
identification of the coefficient spaces and preserves the Fredholm index.
\end{proof}

Under Corollary~\ref{cor:asymptotic-kadec}, let $(g_n)_{n\in\Z}$ be
the unique biorthogonal dual Riesz basis to $(e_{\lambda_n})_{n\in\Z}$.
The standard Riesz-basis reconstruction formulas give
\[
    f=\sum_{n\in\Z}\langle f,g_n\rangle e_{\lambda_n}
     =\sum_{n\in\Z}\langle f,e_{\lambda_n}\rangle g_n,
    \qquad f\in L^2(0,1),
\]
with unconditional convergence of both series. More generally, for every
pseudo-Riesz exponential system in
Theorem~\ref{thm:real-exponential-classification},
Theorem~\ref{thm:automatic-besselness} shows that every finite-defect
reconstruction converging in norm on $L^2(0,1)$ converges unconditionally.
Its reconstructing sequence is pseudo-Riesz with the same synthesis index.
The finite-rank error permits this conclusion even when the exponential
system is incomplete. Since exponentials have no zero terms, the
unconditional-convergence equivalences in
Corollary~\ref{cor:finite-defect-heil-yu} apply whenever the system
admits a finite frame completion.

\bibliographystyle{plain}
\bibliography{Paper_PRB_MB}

\end{document}